\documentclass[11pt]{article}

\usepackage{graphicx}
\usepackage{subfigure}
\usepackage[T1]{fontenc}
\usepackage[sc]{mathpazo}
\usepackage{amsmath}
\usepackage{amsthm,amssymb,amsfonts}
\usepackage{upgreek}
\usepackage[hmargin=1.0in]{geometry}

\begin{document}
\bibliographystyle{plain}
\newcommand{\sgn}{\mathop{\mathrm{sgn}}}

\newcommand{\Sec}[1]{Section~\ref{sect:#1}}
\newtheorem{dfn}{Definition}[section]
\newtheorem{thm}{Theorem}[section]
\newtheorem{lem}[thm]{Lemma}
\newtheorem{cor}[thm]{Corollary}
\theoremstyle{remark}
\newtheorem{rem}{Remark}

\newcommand{\sspcoef}{\mathcal{C}}
\newcommand{\tc}{\tilde{\sspcoef}}
\newcommand{\tp}{\tilde{p}}
\newcommand{\m}[1]{\mathbf{#1}}
\newcommand{\mA}{\m{A}}
\newcommand{\mtA}{\tilde{\mA}}
\newcommand{\matalpha}{\boldsymbol{\upalpha}}
\newcommand{\matalphat}{\tilde{\boldsymbol{\upalpha}}}
\newcommand{\matbeta}{\boldsymbol{\upbeta}}
\newcommand{\bv}{\mathbf{v}}
\newcommand{\bu}{\mathbf{u}}
\newcommand{\Dx}{\Delta x}
\newcommand{\tF}{\tilde{F}}

\newcommand{\mL}{\m{L}}
\newcommand{\mI}{\m{I}}
\newcommand{\tL}{\tilde{\mL}}
\newcommand{\tz}{\tilde{z}}
\newcommand{\tR}{\tilde{R}}
\newcommand{\matgamma}{\boldsymbol{\upgamma}}

\newcommand{\be}{\begin{equation}}
\newcommand{\ee}{\end{equation}}
\newcommand{\Dt}{\Delta t}
\newcommand{\transpose}{^\mathrm{T}}
\newcommand{\dd}[2]{\frac{\partial #1}{\partial #2}}

\newcommand{\tlambda}{\tilde{\lambda}}
\newcommand{\ta}{\tilde{a}}
\newcommand{\tA}{\tilde{\m{A}}}
\newcommand{\tb}{\tilde{b}}
\newcommand{\Oop}{{\cal O}}
\newcommand{\tbeta}{\tilde{\beta}}
\newcommand{\DtFE}{\Dt_{\textup{FE}}}

\title{Step Sizes for Strong Stability Preservation with Downwind-biased
Operators}
\author{
  David I. Ketcheson\thanks{King Abdullah University of Science and Technology,
    Box 4700, Thuwal, Saudi Arabia, 23955-6900.
   (\mbox{david.ketcheson@kaust.edu.sa}).  This publication is based on work
   supported by award No. FIC/2010/05 - 2000000231, made by King Abdullah
   University of Science and Technology (KAUST).}}
\date{}

\maketitle

\abstract{
Strong stability preserving (SSP) integrators for initial value ODEs
preserve temporal monotonicity 
solution properties in arbitrary norms.  All existing SSP methods, including
implicit methods, either
require small step sizes or achieve only first order accuracy.
It is possible to achieve more relaxed step size restrictions in the
discretization of hyperbolic PDEs through the use of both upwind- and
downwind-biased semi-discretizations.
We investigate bounds on the maximum SSP step size 
for methods that include negative coefficients 
and downwind-biased semi-discretizations.
We prove that the downwind SSP coefficient for linear multistep methods
of order greater than one is at most equal to two, while
the downwind SSP coefficient for explicit Runge--Kutta
methods is at most equal to the number of stages of the method.
In contrast, the maximal
downwind SSP coefficient for second order Runge--Kutta methods is shown to be
unbounded.  We present a class of such methods with arbitrarily large SSP
coefficient and demonstrate that they achieve second order accuracy
for large CFL number.
}

\section{Introduction}
  \subsection{Strong stability preservation}
This work is concerned with numerical methods for the initial value problem
\begin{align} \label{ivp}
u'(t) & = F(u) & u(0)=u_0,
\end{align}
where $u\in\Re^n$ and $F : \Re^n \to \Re^n$.
Numerical methods for \eqref{ivp} compute a sequence of solutions
$u^1,u^2,\dots$ approximating the true solution at times $\Dt,2\Dt,\dots$.
We are interested in numerical solutions that satisfy the monotonicity 
property
\begin{align} \label{onestep-monotonicity}
\|u^{n}\| \le \|u^{n-1}\|
\end{align}
in the case of a one-step method, or more generally
\begin{align} \label{monotonicity}
\|u^{n}\| \le \max\left\{\|u^{n-1}\|,\|u^{n-2}\|,\dots,\|u^{n-k}\|\right\}
\end{align}
for a $k$-step method.  

A common approach is to assume that the initial value problem \eqref{ivp}
is monotone under
forward Euler integration, subject to some step size restriction:
\begin{align} \label{FEcond}
\|u + \Dt F(u)\| \le \|u\| \mbox{ for all } u, \mbox{ for } 0\le\Dt\le\DtFE.
\end{align}
The term {\em strong stability preserving} (SSP) is used to denote any method
that gives a solution satisfying the monotonicity condition \eqref{monotonicity} 
whenever applied to a initial value problem \eqref{ivp} satisfying the forward Euler
condition \eqref{FEcond}.  This can be shown to hold under a step size
restriction of the form
\begin{align}
\Dt \le \sspcoef \DtFE,
\end{align}
where the factor $\sspcoef$, referred to as the {\em SSP coefficient},
depends only on the numerical method.  For a recent review of SSP methods
see \cite{gottlieb2009}.

The SSP coefficient is, for most known methods, not very large; for 
a broad class of explicit general linear methods it is never greater than 
the number of stages of the method \cite{ketcheson2009a}, and for many 
classes of implicit methods it is known or 
conjectured to be no greater than twice the number of stages 
\cite{hundsdorfer2005,ferracina2008,ketcheson2009}.
No known methods have $\sspcoef>2s$ and exhibit higher than first
order convergence for large CFL numbers.

\subsection{Downwinding}
The study of SSP methods has been motivated by the numerical solution of
hyperbolic conservation laws; in one dimension these take the form
\begin{align} \label{conslaw}
U_t + f(U)_x = 0.
\end{align}
Semi-discretization of the conservation law \eqref{conslaw} leads to the
initial value problem \eqref{ivp}, where
$u$ is a finite-dimensional approximation of $U$ and 
$F(u)$ is an approximation to $-f(U)_x$.

The solution of the scalar conservation law \eqref{conslaw} has the property
that its total variation does not increase in time.  Hence it is
desirable for a numerical discretization to satisfy \eqref{monotonicity}
with respect to the total variation semi-norm; such schemes are said to be
total variation diminishing (TVD).  A common approach to development of
TVD methods is to employ
a semi-discretization for which the TVD property holds under forward
Euler integration in time, up to some maximal time step size; this
is just condition \eqref{FEcond}, where $\|\cdot\|$ is taken to be
the total variation semi-norm.

Such semi-discretizations generally are upwind-biased.  By considering
corresponding downwind-biased semi-discretizations, one arrives at
an operator $\tF$ that approximates $f(U)_x$ and satisfies
\begin{align} \label{dwFEcond}
\|u + \Dt \tF(u)\| \le \|u\| \mbox{ for all } \Dt\le\DtFE.
\end{align}

For example, consider the advection equation
\begin{align} \label{eq:advect}
U_t + U_x & = 0
\end{align}
discretized via the upwind and downwind discretizations:
\begin{align} \label{eq:advupdown}
u_i'(t) & = -\frac{u_i-u_{i-1}}{\Dx} = F_i(u) & u_i'(t) & =  -\frac{u_{i+1} - u_i}{\Dx} = -\tF_i(u).
\end{align}
Here $u_i(t)$ is an approximation to $U(i\Dx,t)$.
In this case the conditions \eqref{FEcond}, \eqref{dwFEcond} hold with
$\DtFE=\Dx$, corresponding to the usual CFL condition.

We say a method is strong stability preserving with {\em downwind
SSP coefficient} $\tc$ if the method
gives a solution satisfying \eqref{monotonicity} whenever 
applied to a system satisfying the forward Euler conditions \eqref{FEcond} 
and \eqref{dwFEcond} under the time step restriction
\begin{align}
\Dt \le \tc \DtFE.
\end{align}

The idea of using downwinding to achieve the TVD property under larger
timesteps for explicit time discretizations was originally introduced 
by Shu and Osher \cite{shu1988,shu1988b}.  These methods are frequently
referred to as {\em methods with downwind biased discretizations}; in the
present work we will refer to them simply as {\em downwind methods} for 
brevity, with the understanding that they incorporate both upwind- and
downwind-biased discretizations.
Optimal explicit downwind Runge--Kutta methods of up to fifth order and ten 
stages were developed in
\cite{ruuth2004}.  Further optimal explicit Runge--Kutta and linear multistep
schemes, along with an approach to
efficient implementation of upwind and downwind WENO discretizations, are given
in \cite{gottlieb2006}.  A theory of necessary and sufficient
conditions for Runge-Kutta methods with downwinding to be SSP was developed in
\cite{higueras2005a}.

The principal reason for studying downwind methods is that the downwind 
SSP coefficient $\tc$ is not as restricted as the SSP coefficient
$\sspcoef$; for instance, explicit Runge-Kutta methods can have order
greater than four and $\tc>0$ \cite{ruuth2004}.  In the present work, we 
consider also {\em implicit} downwind methods.  Generally speaking, implicit 
numerical
methods are used in order to allow the use of timesteps based solely on 
accuracy considerations and not on stability.  
As discussed above, implicit SSP methods require step sizes not much larger
than those allowed by explicit methods.  The principal aim of the present
is to determine whether implicit downwind SSP methods allow much larger step 
sizes.

\subsection{Downwind Methods} \label{sect:add}
Downwind Runge-Kutta methods take the form
    \begin{subequations} \label{eq:add-runge-kutta} \begin{align}
    \label{eq:add-RKstages}
    y_i &= u^{n-1} + \Dt \sum_{j=1}^s a_{ij} F \left(t^{n-1} + c_j \Dt, 
    y_j \right)+\Dt \sum_{j=1}^s \ta_{ij} \tF \left(t^{n-1} + c_j \Dt, 
    y_j \right), \ \ \ 1 \leq i \leq s \\
    u^{n} & = u^{n-1} + \Dt \sum_{j=1}^s b_{j} F \left(t^{n-1} + c_j 
    \Dt, y_j \right) + \Dt \sum_{j=1}^s \tb_{j} \tF \left(t^{n-1} + c_j 
    \Dt, y_j \right).
    \end{align} \end{subequations}
The method is explicit if each stage depends only on previous stages;
i.e., if $a_{ij}=0,\ta_{ij}=0$ for $j\ge i$.

Downwind linear multistep methods take the form
\be \label{eq:add_lmm_dw}
u_n - \Dt\beta_k F(u_n) - \Dt\tbeta_k \tF(u_n) = 
  \sum_{j=0}^{k-1} \alpha_j u_{n-k+j}+ \Dt \beta_j F(u_{n-k+j}) + \Dt\tbeta_j\tF(u_{n-k+j}).
\ee
The method is explicit if $\beta_k=\tbeta_k=0$.

%

\section{Summary of main results}
  In this section we present the main results of the paper.
The proofs are deferred to later sections.

Our first result concerns explicit downwind Runge--Kutta methods.  It is
known that the SSP coefficient $\sspcoef$ for a broad class of general 
linear methods (without
downwinding) cannot be greater than the number of stages \cite{ketcheson2009a}.
It turns out that the same bound holds for explicit downwind Runge--Kutta methods.

\begin{thm} \label{dwrmaxthm}
The downwind SSP coefficient $\tc$ of any first order accurate explicit 
downwind Runge--Kutta method is no greater than the number of stages of the method.
\end{thm}

It is already known (see \cite{ketcheson2009a}) that the bound
$\tc\le1$ holds for explicit linear multistep methods.
Together with the theorem above, this implies that not too much can be 
gained by using downwinding in explicit methods, at least in an asymptotic
sense.  Hence we consider implicit downwind
methods.  For implicit linear multistep methods of greater than first order
accuracy (without downwinding),
it is known that the SSP coefficient cannot be greater than two.
It turns out that the same bound holds for implicit downwind linear
multistep methods.

\begin{thm} \label{thm:dwlmm}
The downwind SSP coefficient of any second order accurate
linear multistep method is at most two.
\end{thm}

Given the negative nature of the two theorems above, one might expect that
bounds on the SSP coefficient hold for the downwind SSP coefficient 
for all classes of methods.  It seems, however, that this is not the case;
for implicit Runge-Kutta methods
it is conjectured that the SSP coefficient can be no greater than $2s$,
where $s$ is the number of stages.  However, it turns out that implicit
downwind Runge-Kutta methods can have arbitrarily large SSP coefficient.

\begin{thm} \label{thm:dwirk}
For any positive finite $r$, there exists a two-stage, second-order 
downwind Runge-Kutta method with downwind SSP coefficient equal to $r$.
\end{thm}

The rest of the paper proceeds as follows.  The proof of Theorem
\ref{dwrmaxthm} is presented in section \ref{sect:dwerk}.  The proof
of Theorem \ref{thm:dwlmm} is presented in \ref{sect:dwilmm}.
In Section \ref{sect:dwirk}, we prove Theorem \ref{thm:dwirk} by
construction and analyze the resulting methods.  Some final remarks
on these results are provided in Section \ref{sect:discuss}.

\section{Explicit Runge--Kutta methods\label{sect:dwerk}}
  In this section, we prove Theorem \ref{dwrmaxthm}, which bounds the downwind
SSP coefficient for explicit downwind Runge--Kutta methods.  SSP methods
are most conveniently analyzed by writing them in a certain Shu-Osher form;
see \cite{spijker2007,ketchesonphdthesis}.  The corresponding form for for
downwind methods is introduced in the following Lemma, which is an easy
corollary of Theorems 3.5 and 3.6 of \cite{spijker2007}.
\begin{lem}
The SSP coefficient of the downwind Runge-Kutta method 
\eqref{eq:add-runge-kutta} is the largest $r\ge0$ such that the method 
can be written in the form
\begin{subequations}\label{dwrk_sspform}
\begin{align} 
y_i & = u^{n-1} 
         + \sum_{j=1}^{s} p_{ij} \left(y_j + \frac{\Dt}{r}F(y_j)\right)
         + \sum_{j=1}^{s} \tp_{ij} \left(y_j + \frac{\Dt}{r}\tF(y_j)\right) \\
u^{n} & = u^{n-1} 
         + \sum_{j=1}^{s} p_{s+1,j} \left(y_j + \frac{\Dt}{r}F(y_j)\right)
         + \sum_{j=1}^{s} \tp_{s+1,j} \left(y_j + \frac{\Dt}{r}\tF(y_j)\right)
\end{align}
\end{subequations}
with all coefficients non-negative:
\begin{align} \label{poscoeff}
p_{ij},\tp_{ij}\ge 0.
\end{align}
\end{lem}

In the remainder of this section, we consider explicit methods of the form
\eqref{dwrk_sspform}; i.e., those for which 
\begin{align} \label{erk}
p_{ij}=\tp_{ij}=0 \mbox{ for all } j\ge i.
\end{align}
Consider the application of a method satisfying \eqref{dwrk_sspform}-\eqref{erk}
to a linear problem with
\begin{align} \label{linode}
F(u) & = \mL u & \tF(u) & = \tL u,
\end{align}
where $\mL,\tL$ are fixed matrices,
and define $z=\Dt\mL,\tz=\Dt\tL$.  Then a straightforward calculation
shows that the solution can be written as $u_n = \psi(z,\tz)u_{n-1}$ where
\begin{align} \label{dwtaylor1}
\psi(z,\tz) & = \sum_{j=0}^s \sum_{l=0}^j \gamma_{jl} 
  \left(1+\frac{z}{r}\right)^{j-l} \left(1+\frac{\tz}{r}\right)^{l} 
  \quad \quad \mbox{with } \gamma_{jl} \ge 0 \mbox{ if } 0\le r\le \tc.
\end{align}

Theorem \ref{dwrmaxthm} follows by considering the conditions for
first-order accuracy of the method and positivity of the coefficients
in the form \eqref{dwtaylor1}.
\begin{proof}[Proof of Theorem \ref{dwrmaxthm}.]
For a given downwind Runge--Kutta method, consider \eqref{dwtaylor1} with $r=\tc$.
Since $\tF(u)\approx-F(u)$, the method approximates the 
true solution to order $p$ if 
$$\psi(z,-z)=\exp(z)+\Oop(z^{p+1}).$$
In terms of the coefficients $\gamma_{jl}$, the conditions for consistency and
first order accuracy are:
\begin{align*}
& \sum_{j=0}^s \sum_{l=0}^j \gamma_{jl} = 1,     \\
& \sum_{j=0}^s \sum_{l=0}^j \gamma_{jl} \left(j-2l\right) = \tc.
\end{align*}
Since the $\gamma$'s are positive, we have
\begin{align*}
\tc & = \sum_{j=0}^s \sum_{l=0}^j \gamma_{jl} \left(j-2l\right)
   \le \max_{j\in(0,s),l\in(0,j)} \left(j-2l\right) 
   = s.
\end{align*}
\end{proof}

\begin{rem}
It is possible to show in a similar way that the bound $\tc\le s$ holds even 
for the broad class of general linear methods considered in \cite{ketcheson2009a}.
\end{rem}

\begin{rem}
In \cite{ketchesonphdthesis}, tighter bounds on $\tc$ were computed
for specific classes of downwind Runge--Kutta methods.  The same approach
used there could be applied to find tighter bounds for classes of
general linear methods.
\end{rem}

\section{Implicit linear multistep methods\label{sect:dwilmm}}
  Lenferink \cite{lenferink1991} showed that the SSP coefficient $\sspcoef$
is no greater than two for implicit linear multistep methods of order greater
than one.  This 
bound was shown to hold in an even more general sense in \cite{hundsdorfer2005}.
In this section we consider the downwind SSP coefficient $\tc$ for 
implicit downwind linear multistep methods
\eqref{eq:add_lmm_dw}.  Our main result, presented already as Theorem
\ref{thm:dwlmm}, states that the same bound proved by Lenferink also 
holds for downwind methods.
Even though this is a more general result than those of 
\cite{lenferink1991,hundsdorfer2005}, we provide a simpler proof.

The method \eqref{eq:add_lmm_dw} is accurate to order $p$ if
\begin{align} \label{oc1}
\sum_{j=0}^{k-1} \alpha_j j^i + \sum_{j=0}^k (\beta_j -\tbeta_j) i j^{i-1} & = k^i & (0\le i\le p)
\end{align}
and has downwind SSP coefficient $\tc\ge r$ if and only if
\begin{subequations} \label{eq:x}
\begin{align} 
\label{subeq1}
\beta_j, \tbeta_j & \ge 0 & (0\le j \le k-1) \\
\alpha_j - r (\beta_j +\tbeta_j) & \ge 0 & (0\le j \le k-1).
\end{align}
\end{subequations}

The following lemma, which appeared in \cite{ketcheson2009a}, 
facilitates the proof of Theorem \ref{thm:dwlmm}.  We say a method
\eqref{eq:add_lmm_dw} of $k$ steps and order $p$
is {\em optimal} if there exists no other method
of at most $k$ steps and order at least $p$ with larger downwind SSP
coefficient $\tc$.

\begin{lem} \label{betalem}
Any optimal downwind linear multistep method \eqref{eq:add_lmm_dw} has the 
property that $\beta_j\tbeta_j=0$ for each $j$.
\end{lem}
\begin{proof}[Proof of Lemma \ref{betalem}]
Note that the order conditions \eqref{oc1} depend only on 
the difference
$\beta_j-\tbeta_j$, while the inequality constraint \eqref{subeq1} can be 
written as (setting $r=\tc$)
$$\tc \le \alpha_j/(\beta_j+\tbeta_j).$$  Suppose that an optimal method has
$\beta_j>\tbeta_j>0$ for $j\in J_1 \subset(0,1,\dots,k-1)$ and
$\tbeta_j>\beta_j>0$ for $j\in J_2 \subset(0,1,\dots,k-1)$.
Then for $j\in J_1$ define $\beta^*_j=\beta_j-\tbeta_j$ and $\tbeta^*_j=0$;
for $j\in J_2$ define $\tbeta^*_j=\tbeta_j-\beta_j$ and $\beta^*_j=0$.  Then
the coefficients obtained by replacing $\beta,\tbeta$ with $\beta^*,\tbeta^*$
satisfy \eqref{eq:x} and \eqref{oc1} with a larger value of
$\tc$, which is a contradiction. 
\end{proof}

We now prove our main result on linear multistep methods.  Similar
to the proof of Theorem \ref{dwrmaxthm}, this result follows from
combining the order conditions and the positivity of the coefficients.

\begin{proof}[Proof of Theorem \ref{thm:dwlmm}]
It is sufficient to prove that $\tc\le 2$ for any optimal method.
Hence we consider an optimal method that is at least second order accurate.
Applying Lemma \ref{betalem}, we define $\beta_j^*=\beta_j + \tbeta_j$ and
$\sigma_j = \sgn(\beta_j - \tbeta_j)$ so that 
$\beta_j-\tbeta_j = \sigma_j\beta_j^*$.  Additionally define 
$\delta_j = \alpha_j - r \beta_j^*$.  Then the order conditions for
order two can be written (taking $r=\tc$)
\begin{subequations}
\begin{align}
& \sum_{j=0}^{k-1} \delta_j + \tc \beta_j^* = 1 \label{lmmoc0} \\
& \sum_{j=0}^{k-1} j\delta_j + (\tc j + \sigma_j) \beta_j^* 
        = k-\sigma_k\beta_k^* \label{lmmoc1} \\
& \sum_{j=0}^{k-1} j^2\delta_j + (\tc j^2 + 2j\sigma_j) \beta_j^* 
        = k(k-2\sigma_k\beta_k^*). \label{lmmoc2}
\end{align}
\end{subequations}
Taking $-k^2\times$\eqref{lmmoc0}$+2k\times$\eqref{lmmoc1}$-$\eqref{lmmoc2}
gives
\begin{align}
\sum_{j=0}^{k-1} -(k-j)^2\delta_j + \left(-\tc(k-j)^2 + 2\sigma_j(k-j)\right) \beta_j^*=0.
\end{align}
Since all coefficients $\delta_j,\beta_j^*$ are non-negative, at least
one of the terms multiplying them must be non-negative, for some value of
$j$ (else the sum would be negative).  Since the terms multiplying 
$\delta_j$ are all negative, this implies that
$$-\tc(k-j)^2 + 2\sigma_j(k-j)\ge0$$
for some $j$; i.e.
$$\tc \le \max_{j\in[0,k-1]} \frac{2\sigma_j(k-j)}{(k-j)^2} \le 2.$$
\end{proof}

\begin{rem}
Using the order conditions and the SSP conditions, it is straightforward
to determine (for any particular $k,p$) the methods with largest 
$\tc$.  This was done in \cite{ketcheson2009a}.
\end{rem}

\section{Implicit Runge--Kutta methods\label{sect:dwirk}}
  It is well known (see, e.g. \cite{spijker1983,kraaijevanger1991,gottlieb2009})
that (implicit) Runge-Kutta methods of order higher than one cannot have
$\sspcoef=\infty$.  Furthermore, it is conjectured that they cannot have
$\sspcoef>2s,$ where $s$ is the number of stages \cite{ferracina2008,ketcheson2009}.

In \cite{higueras2005a}, the following question was posed: do there exist
high order implicit downwind Runge-Kutta methods with $\tc=\infty$?  
Although we have not answered this question directly, we have found downwind
methods with arbitrarily large $\tc$.  One class of such methods is the two-stage, 
second-order family
\begin{subequations} \label{eq:optdwrk}
\begin{align}
y_1 & = \frac{2}{r(r-2)} u^{n-1} 
         + \frac{2}{r}\left(y_1+\frac{\Dt}{r}F(y_1)\right)
         + \frac{r^2-4r+2}{r(r-2)}\left(y_2+\frac{\Dt}{r}\tF(y_2)\right) \\
y_2     & = y_1 + \frac{\Dt}{r}F(y_1) \\
u^{n} & = y_2 + \frac{\Dt}{r}F(y_2).
\end{align}
\end{subequations}
Here we have written the method in the form \eqref{dwrk_sspform} so that
the downwind SSP coefficient is apparent: $\tc=r$, for any $r>2+\sqrt{2}$.
It can be shown that the method \eqref{eq:optdwrk} is A-stable (i.e., 
unconditionally stable when $F(u)=-\tF(u)=\lambda u$ with $\Re(\lambda)\le0$).

Numerical searches indicate that other two-stage, second-order methods with 
large $\tc$ exist.  An obvious drawback of
method \eqref{eq:optdwrk} is that it is fully implicit.  However, a search for
two-stage, second-order diagonally implicit methods yielded no results.

For reference, we include here also the coefficients of the method when
written in form \eqref{eq:add-runge-kutta}:
\begin{align} \label{pbc}
\begin{array}{c|cc|cc} 
  \frac{r-2}{r} & \frac{r^2-2r-2}{2r} & 0 & 0 & \frac{r^2-4r+2}{2r} \\
  \frac{r-1}{r} & \frac{r-2}{2}       & 0 & 0 & \frac{r^2-4r+2}{2r} 
              \\ \hline & \frac{r-2}{2} & \frac{1}{r} & 0 & \frac{r^2-4r+2}{2r} \end{array}.
\end{align}
The left part of \eqref{pbc} corresponds to the usual Butcher coefficients,
and the rightmost part displays the additional coefficients $\ta_{ij},\tb_j$.
In the case $F=-\tF$, the method reduces to an ordinary Runge-Kutta method; we
refer to this method as the {\em underlying method}:
\begin{align} \label{pbcu}
\begin{array}{c|cc} 
  \frac{r-2}{r} & \frac{r^2-2r-2}{2r} & -\frac{r^2-4r+2}{2r} \\
  \frac{r-1}{r} & \frac{r-2}{2}       & -\frac{r^2-4r+2}{2r} 
              \\ \hline & \frac{r-2}{2} & \frac{4-r}{2} \end{array}.
\end{align}
Although the downwind method \eqref{pbc} will in general behave differently from the underlying
method \eqref{pbcu}, analysis of the latter may still give useful insight.
Straightforward calculation shows that the stability function of method \eqref{pbcu} is
\begin{align} \label{stabfunc}
\psi(z) & = \frac{1 + \frac{2}{r} z + \frac{1}{r^2} z^2}{1 - \frac{r-2}{r} z + \frac{r^2-4r+2}{2r^2}z^2}.
\end{align}
Further calculation reveals that the method is A-stable since $|\psi(z)|\le 1$ for all $z$ with
non-positive real part.  It is clearly not L-stable, but it is nearly so for large $r$ since
\begin{align*}
\lim_{|z| \to \infty} \psi(z) & = \frac{2}{r^2 - 4r + 2}
\end{align*}
is quite small for large $r$.


\section{Numerical tests}
In this section we conduct a few numerical tests to study the accuracy of the family of
methods introduced in the last section.  We compare the backward Euler method,
the implicit trapezoidal Runge-Kutta method, and the downwind Runge-Kutta
method \eqref{eq:optdwrk}.  In all tests we take $r=8$.

In the first test we will see that the temporal error for the downwind method includes
a term related to the spatial discretizations.  Consequently, some care must be
taken when choosing the spatial discretizations.  The second and third tests
demonstrate that the method performs well when appropriate spatial discretizations 
are chosen.

\subsection{First-order upwind advection}
As a first test we solve the advection equation \eqref{eq:advect} using the
upwind and downwind differences \eqref{eq:advupdown} with $\Dx=1/128$.  We
consider the domain $0\le x\le 1$ with periodic boundary conditions and initial
condition $u(x,0)=1-H(x-1/2)$ where $H(x)$ is the Heaviside function.  In
Figure \ref{fig:squareadvect} we plot the solutions at $t=1$.  

Figure \ref{figa} shows results obtained with $\Dt=0.9\Dx$.  As expected,
backward Euler is more dissipative than the second order trapezoidal method.
Surprisingly, the (second order accurate) downwind Runge--Kutta method is more 
dissipative than even the first-order backward Euler method.
Figure \ref{figb} shows results obtained with $\Dt=8.0\Dx$, the SSP limit
for the downwind method.  In this case, the trapezoidal method, which has SSP
coefficient $\sspcoef=2$, generates overshoots.  The backward Euler method
is much more dissipative for this large CFL number.  The downwind method
is more accurate than backward Euler and maintains maximum norm monotonicity,
but the high level of dissipation exhibited is still disappointing.

\begin{figure}
\subfigure[CFL=0.9\label{figa}]{\includegraphics[width=3in]{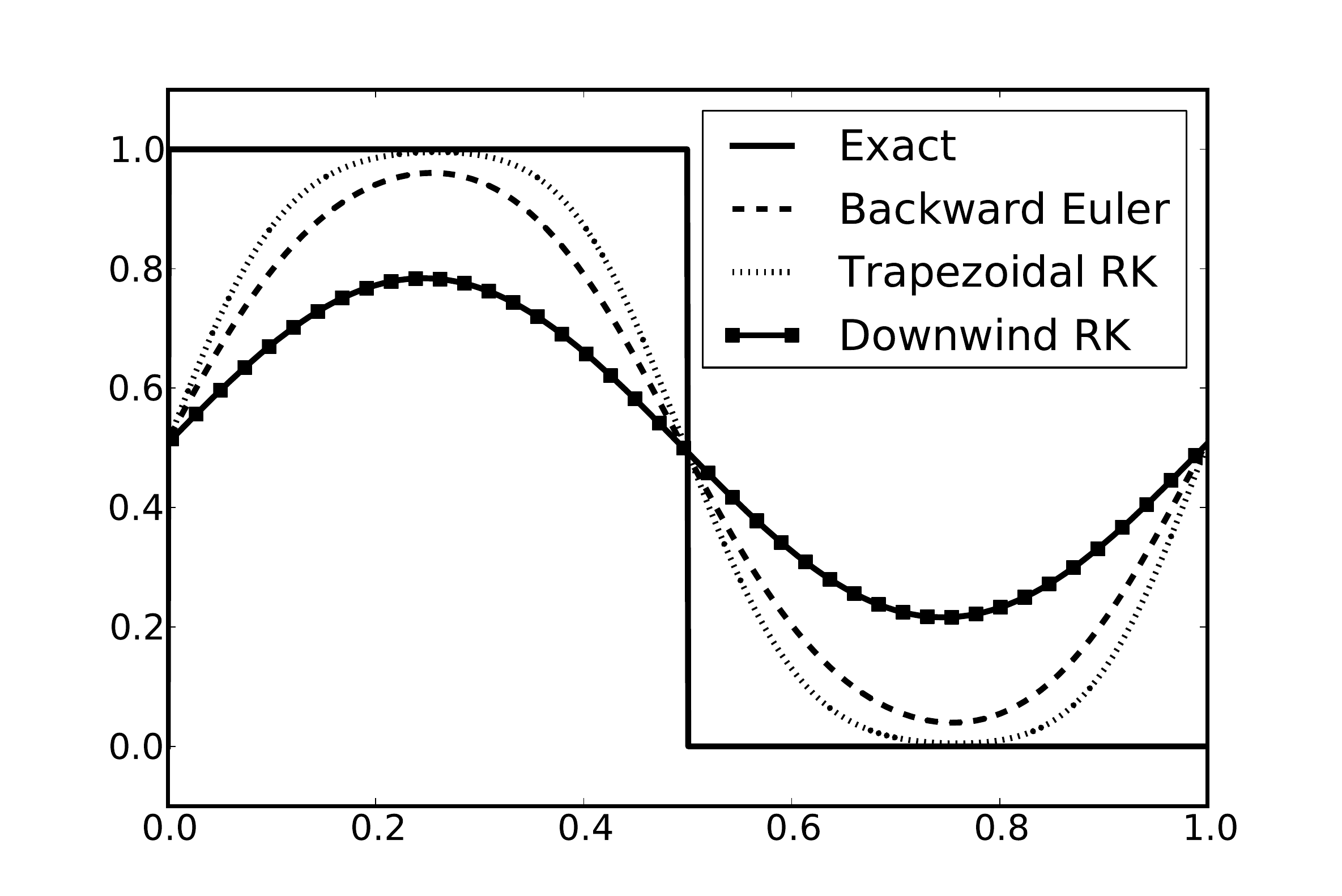}}
\subfigure[CFL=8.0\label{figb}]{\includegraphics[width=3in]{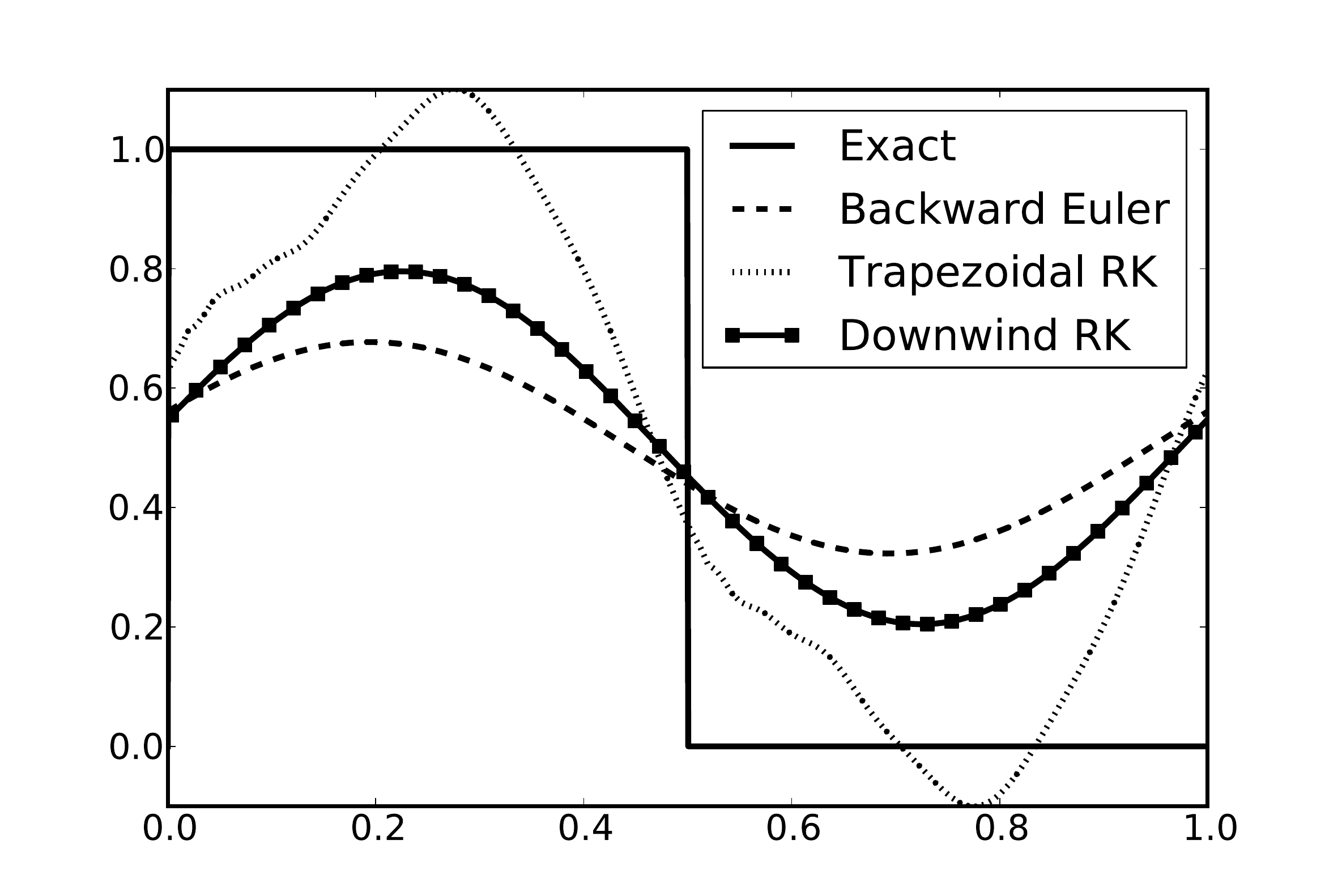}}
\caption{Square wave advection with the downwind Runge-Kutta method 
\eqref{eq:optdwrk}, compared to two standard methods.\label{fig:squareadvect}}
\end{figure}

The behavior of the downwind method can be understood through the following 
analysis.  Write the upwind and downwind discretizations as 
$F(u)=\mL u, \tF(u)=-\tL u$, where $\mL,\tL$ are the matrices
\begin{align}
\mL & = \frac{1}{\Dx}\begin{pmatrix}
            -1 & & & 1\\
            1 & -1 \\
            & \ddots & \ddots \\
            & & 1 & -1
        \end{pmatrix} &
\tL & = \frac{1}{\Dx}\begin{pmatrix}
            1 & -1 & & \\
            & 1 & -1 \\
            & & \ddots & \ddots \\
            -1 & & & 1
        \end{pmatrix}
\end{align}
Any downwind Runge--Kutta method applied to this problem 
results in a recurrence of the form (see section 3 above)
\begin{align}
u^n & = \psi(\Dt \mL, \Dt \tL) u^{n-1}.
\end{align}
For the method \eqref{eq:optdwrk}, we find
\begin{align*}
\psi(\Dt\mL,\Dt\tL) & = \left(1+\frac{z}{r}\right)^2 
            \left(1+\frac{r^2-4r+2}{2r}\Dt\tL - \frac{r^2-2r-2}{2r}\Dt\mL
                    - \frac{r^2-4r+1}{2r^2}\Dt^2 \mL\tL\right)^{-1} \\
    & = 1 + \Dt \mL + \frac{r^2-4r+2}{2r}\Dt\left(\mL - \tL\right) + \Oop(\Dt^2).
\end{align*}
Since $\mL u \approx -u_x$, the first two terms ensure that the method is at
least first order accurate.  However, the term involving $\left(\mL - \tL\right)$ is problematic, since
\begin{align}
\mL - \tL & = \frac{1}{\Dx}\begin{pmatrix}
            -2 & 1 & & 1\\
            1 & -2 & 1 \\
            & \ddots & \ddots & \ddots \\
            1 & & 1 & -2
        \end{pmatrix}
        \approx \Dx \ u_{xx}.
\end{align}

Thus, if $F,\tF$ are spatial discretizations with order of accuracy $q$,
then the one-step error for the downwind Runge--Kutta method \eqref{eq:optdwrk}
will contain a diffusive term of $\Oop(r\Dx^q\Dt)$.  In order to avoid loss of 
accuracy, the spatial discretization should ensure that $r\Dx^q\lessapprox \Dt$,
so that this term is no larger than the $\Oop(\Dt^2)$ term.

\subsection{WENO advection}
To demonstrate that proper accuracy is acheived when the foregoing condition is
satisfied, we use a fifth-order WENO interpolation to determine the fluxes.  
Table \ref{tbl:conv}
shows convergence results in the maximum norm for advection of a sine wave
($u(x,0)=\sin(2\pi x)$ using a CFL number of 8.  Notice that the downwind
method achieves its design order of 2, and gives accuracy similar to the
trapezoidal Runge--Kutta method.  
It should be noted that the downwind method is more expensive
computationally, since it is not diagonally implicit and requires both upwind
and downwind operator evaluations. However, the purpose of this
test is only to confirm the theoretically predicted accuracy.

\begin{table}
\center
\begin{tabular}{l|ccc}
N  & Backward Euler & 2nd order RK & Downwind RK \\ \hline
32 & 0.728& 0.730& 0.436\\
64 & 0.603& 0.215 & 0.168\\
128 & 0.452& 0.054 & 0.043\\
256 & 0.292& 0.013& 0.011\\
\end{tabular}
\caption{Max-norm errors for advection of a sine wave using fifth-order WENO
discretization and a CFL number of 8.\label{tbl:conv}}
\end{table}

\subsection{Burgers equation}
In order to provide an initial assessment of the suitability of these
methods for application to nonlinear hyperbolic conservation laws, consider
Burgers equation:
$$U_t + \left(\frac{1}{2} U^2\right)_x = 0$$
on the unit interval with initial condition $U(x,0)=\sin(2\pi x)$ and periodic boundary
conditions.  Fifth-order WENO interpolation is again used to determine the fluxes.
We consider the solution at time $t=0.16$, just after a shock has formed.  
Figure \ref{burgers65} shows a closeup around the shock for the exact solution
(obtained using characteristics) along with solutions computed by each of the
methods compared above, using CFL number $6.5$ in each case.  A grid with $\Dx=1/512$
points is used.  As expected, the
backward Euler method is much more dissipative than the second order methods.
Interestingly, the trapezoidal method solution does not show oscillations {\em
per se}, but clearly has an unphysical bump behind the shock.  The downwind 
method is by far the most accurate.  Again, it should be noted that the downwind
method is more expensive, since it involves an implicit solve of twice as many equations.

\begin{figure} \begin{center}
\includegraphics[width=4in]{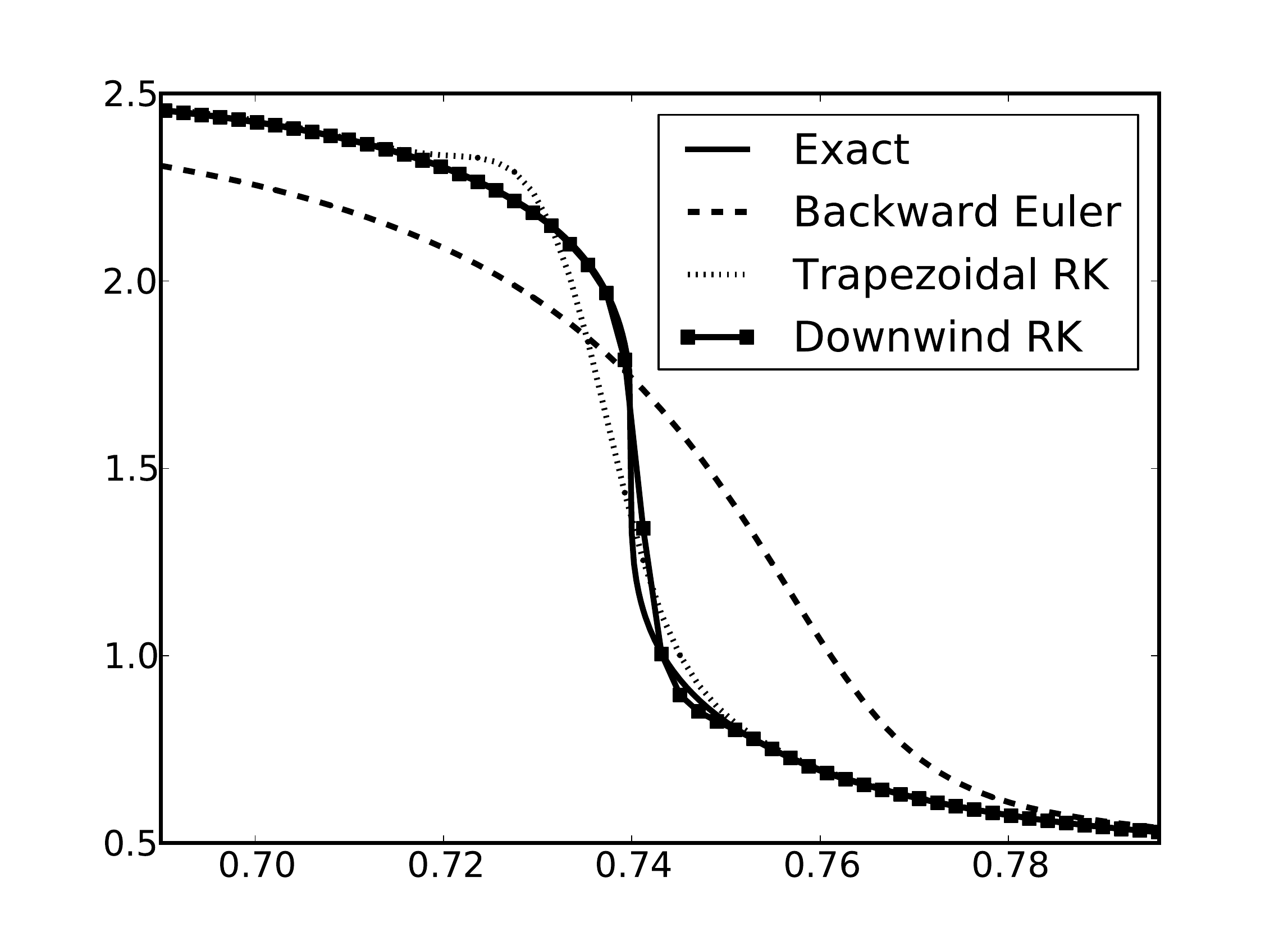}
\caption{Closeup view of shock in solution of Burgers equation with various methods using a CFL number of 6.5.\label{burgers65}}
\end{center}
\end{figure}

Figure \ref{burgerscfl} shows a comparison of two solutions obtained with the downwind
RK method \ref{eq:optdwrk} using two different CFL numbers.  The solutions are almost
indistinguishable, suggesting that the spatial error is dominant, as one might expect. 
Close inspection reveals that the solution using the larger CFL
number is more accurate.  Importantly, it seems that the method does not become more
dissipative at large CFL numbers.

\begin{figure} \begin{center}
\includegraphics[width=4in]{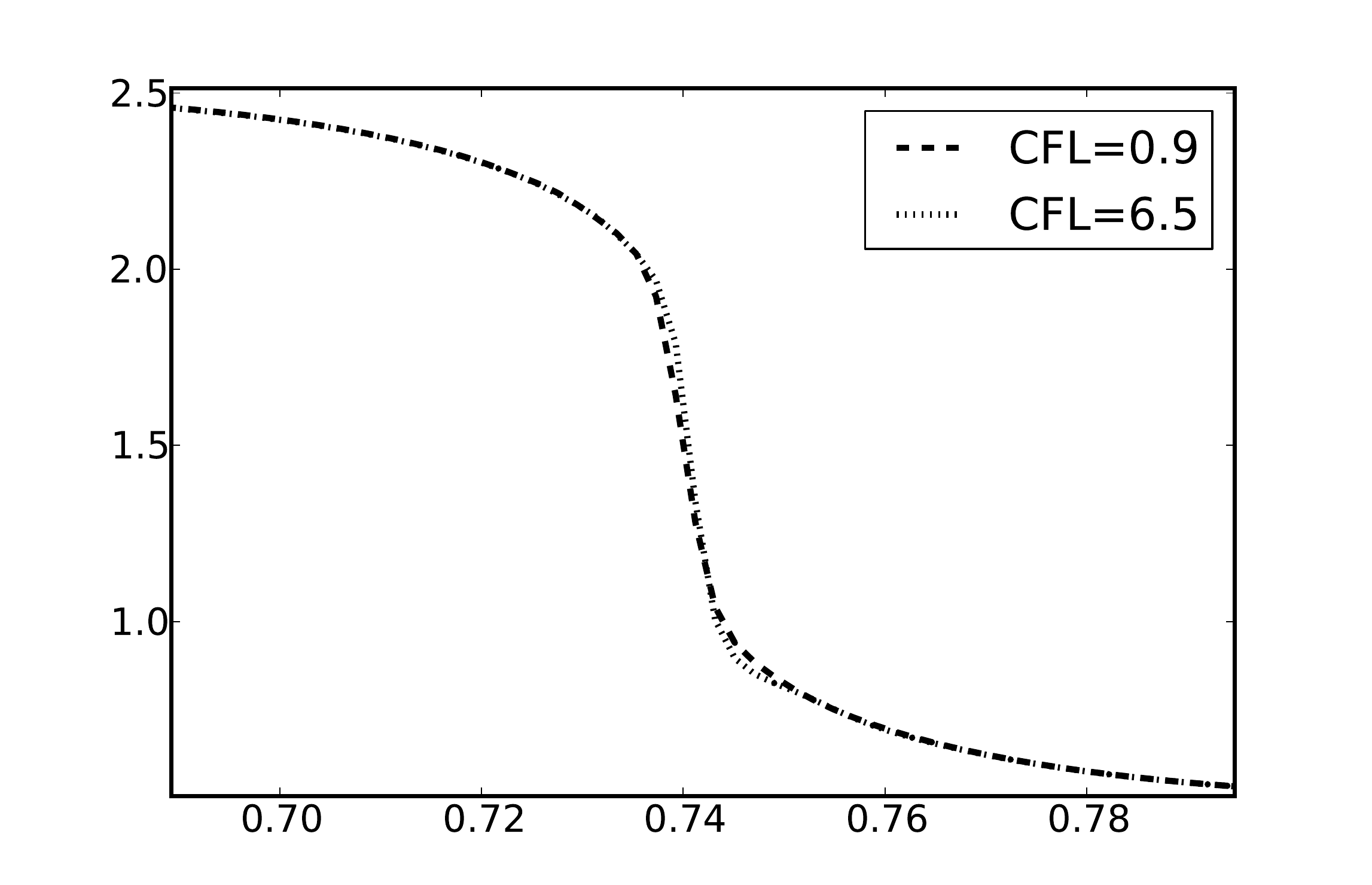}
\caption{Closeup view of shock in solution of Burgers equation obtained with the 
downwind RK method \eqref{eq:optdwrk} with $r=8$
and using two different CFL numbers.\label{burgerscfl}}
\end{center}
\end{figure}

We note in passing that solution of the nonlinear system of equations for 
implicit WENO schemes is a significant challenge, especially for large CFL number and
in the presence of shocks.
The CFL number 6.5 used here allowed the use of an easily accessible nonlinear solver,
namely the \verb newton_krylov  and \verb fsolve  functions of the SciPy package.
Efficient solution of the nonlinear system for larger CFL numbers is an area for
future research.




\section{Discussion\label{sect:discuss}}
  The new bounds on the downwind SSP coefficient proven here for 
explicit Runge--Kutta methods and implicit linear multistep methods
are disappointing, since they indicate that nothing can be gained
-- in an asymptotic sense -- by including downwind-biased discretizations
for integrators in these classes.
Previous attempts to find methods that have large SSP coefficients have
been similarly disappointing \cite{macdonald2007,ketcheson2009}.  

The family of second order implicit downwind Runge--Kutta methods we report
here is therefore quite remarkable.  These are the first methods known to have
large SSP coefficient and to provide higher than first order accuracy in
practice for CFL numbers larger than unity.  They may be useful for problems
with hyperbolic components requiring large time steps and the avoidance of
spurious oscillations.

In light of this discovery, the following questions are naturally of interest:
\begin{itemize}
  \item Do higher (than 2nd) order downwind Runge--Kutta methods exist with large
        downwind SSP coefficient?
  \item Do diagonally implicit downwind Runge--Kutta methods exist with
        large SSP coefficient?
  \item Can these methods be implemented efficiently in combination with high 
        order space discretizations and large time steps, for problems with shocks?
\end{itemize}
Preliminary investigation of the first question indicates the answer is affirmative.
Work on all three questions is ongoing and will be presented elsewhere.

\bibliography{dwssp}

\begin{thebibliography}{10}

\bibitem{ferracina2008}
Luca Ferracina and Marc Spijker.
\newblock Strong stability of singly-diagonally-implicit {R}unge-{K}utta
  methods.
\newblock {\em Applied Numerical Mathematics}, 2008.
\newblock doi:10.1016/j.apnum.2007.10.004.

\bibitem{gottlieb2009}
Sigal Gottlieb, David~I. Ketcheson, and Chi-Wang Shu.
\newblock High order strong stability preserving time discretizations.
\newblock {\em Journal of Scientific Computing}, 38(3):251, 2009.

\bibitem{gottlieb2006}
Sigal Gottlieb and Steven~J. Ruuth.
\newblock Optimal strong-stability-preserving time-stepping schemes with fast
  downwind spatial discretizations.
\newblock {\em Journal of Scientific Computing}, 27:289--303, 2006.

\bibitem{higueras2005a}
I~Higueras.
\newblock Representations of {R}unge-{K}utta methods and strong stability
  preserving methods.
\newblock {\em Siam Journal On Numerical Analysis}, 43:924--948, 2005.

\bibitem{hundsdorfer2005}
Willem Hundsdorfer and Steven~J. Ruuth.
\newblock On monotonicity and boundedness properties of linear multistep
  methods.
\newblock {\em Mathematics of Computation}, 75(254):655--672, 2005.

\bibitem{ketcheson2009a}
David~I Ketcheson.
\newblock {Computation of optimal monotonicity preserving general linear
  methods}.
\newblock {\em Mathematics of Computation}, 78:1497--1513, 2009.

\bibitem{ketchesonphdthesis}
David~I. Ketcheson.
\newblock {\em {High Order Strong Stability Preserving Time Integrators and
  Numerical Wave Propagation Methods for Hyperbolic PDEs}}.
\newblock Ph.d. thesis, University of Washington, 2009.

\bibitem{ketcheson2009}
David~I. Ketcheson, Colin~B. Macdonald, and Sigal Gottlieb.
\newblock Optimal implicit strong stability preserving {R}unge-{K}utta methods.
\newblock {\em Applied Numerical Mathematics}, 52(2):373, 2009.

\bibitem{kraaijevanger1991}
J.~F. B.~M. Kraaijevanger.
\newblock Contractivity of {R}unge-{K}utta methods.
\newblock {\em BIT}, 31:482--528, 1991.

\bibitem{lenferink1991}
H~W.~J. Lenferink.
\newblock Contractivity-preserving implicit linear multistep methods.
\newblock {\em Math. Comp.}, 56:177--199, 1991.

\bibitem{macdonald2007}
Colin Macdonald, Sigal Gottlieb, and Steven Ruuth.
\newblock A numerical study of diagonally split {R}unge-{K}utta methods for
  {PDE}s with discontinuities.
\newblock {\em Journal of Scientific Computing}, 2008.
\newblock doi:10.1007/s10915-007-9180-6.

\bibitem{ruuth2004}
S.~J. Ruuth and R.~J. Spiteri.
\newblock High-order strong-stability-preserving {R}unge-{K}utta methods with
  downwind-biased spatial discretizations.
\newblock {\em SIAM Journal of Numerical Analysis}, 42:974--996, 2004.

\bibitem{shu1988}
C.-W. Shu and S.~Osher.
\newblock Efficient implementation of essentially non-oscillatory
  shock-capturing schemes.
\newblock {\em Journal of Computational Physics}, 77:439--471, 1988.

\bibitem{shu1988b}
Chi-Wang Shu.
\newblock Total-variation diminishing time discretizations.
\newblock {\em SIAM J. Sci. Stat. Comp.}, 9:1073--1084, 1988.

\bibitem{spijker1983}
M.~N. Spijker.
\newblock Contractivity in the numerical solution of initial value problems.
\newblock {\em Numerische Mathematik}, 42:271--290, 1983.

\bibitem{spijker2007}
M.~N. Spijker.
\newblock Stepsize conditions for general monotonicity in numerical initial
  value problems.
\newblock {\em Siam Journal On Numerical Analysis}, 45:1226--1245, 2007.

\end{thebibliography}

\end{document}